\documentclass[11pt]{article}

\usepackage{amssymb,amsmath,amstext,amscd,amsthm,latexsym}

\usepackage[dvipdfmx, bookmarkstype=toc, colorlinks=true, linkcolor=cyan, pdfborder={0 0 0}, bookmarks=true, bookmarksnumbered=true]{hyperref}

\numberwithin{equation}{section}

\theoremstyle{plain}
\newtheorem{theoremA}{Theorem A(Mertens)}

\newtheorem{theoremB}{Theorem B(Mertens)}

\newtheorem{theoremC}{Theorem C(Serre)}

\newtheorem{density}{Density Hypothesis(DH)}

\newtheorem{theorem}{Theorem}
\begin{document}

\title{Variations on theorems of Mertens}

\author{Nobushige Kurokawa\footnote{Department of Mathematics, Tokyo Institute of Technology} \and Hidekazu Tanaka\footnote{6-15-11-202 Otsuka, Bunkyo-ku, Tokyo}}

\date{September 16, 2022}



\maketitle

\begin{abstract}
We present variations on theorems of Mertens as special cases of Density Hypothesis. Moreover, we study a Serre's estimate concerning Lang-Weil estimate.
\end{abstract}

\section{Introduction} In 1874 Mertens \cite{M} proved the following theorems:
\begin{theoremA} \[
\prod_{p \leq t}\biggl(1-\frac{1}{p}\biggl) \sim e^{-\gamma} (\log t)^{-1}
\]
as $t \to \infty$, where $p$ runs over prime numbers.
\end{theoremA}

\begin{theoremB} \[
\prod_{p:odd \; prime} \biggl(1-\frac{(-1)^{\frac{p-1}{2}}}{p}\biggl) = \frac{4}{\pi}.
\]
\end{theoremB}

In this paper we present an interpretation to these theorems as special cases of the following expectation:

\begin{density} Let $X$ be an algebraic variety over the rational number field ${\mathbb Q}$. Define the density function for $t>0$ as
\[
||X||_{t} = \prod_{p \leq t} \frac{|X({\mathbb F}_{p})|}{p^{{\rm dim}(X)}}.
\]
Then there would exist a positive constant $C(X)$ and an integer $r(X)$ satisfying
\[
||X||_{t} \sim C(X) (\log r)^{r(X)} 
\]
as $t \to \infty$.
\end{density}

\begin{theorem} Theorem A and Theorem B are cases
\par (1) $X={\mathbb G}_{m}={\rm GL}(1)$
\\
and
\par (2) $X={\mathcal C}=\{(x,y)|x^2+y^2=1\}$ (circle).
\end{theorem}

\begin{proof}[Proof of Theorem 0]
\par (1) Let $X={\mathbb G}_{m}={\rm GL}(1)$. Then we have 
\[
|X({\mathbb F}_{p})|=|{\rm GL}(1,{\mathbb F}_{p})|=p-1.
\] Thus we have
\begin{align*}
||X||_{t} &= ||{\rm GL}(1)||_{t} = \prod_{p \leq t} \frac{|{\rm GL}(1,{\mathbb F}_{p})|}{p}\\
&= \prod_{p \leq t} (1-p^{-1}).
\end{align*}
\par (2) Let $X=\{(x,y)|x^2+y^2=1\}$. Then we have 
\[
|X({\mathbb F}_{p})|=\left\{
\begin{array}{ccc}
2 & \cdots & p=2,\\
p-1 & \cdots & p \equiv 1 \; {\rm mod} \; 4,\\
p+1 & \cdots & p \equiv 3 \; {\rm mod} \; 4.
\end{array}
\right.
\] Thus we have
\begin{align*}
||X||_{t} &= \prod_{p \leq t} \frac{|X({\mathbb F}_{p})|}{p}\\
&\sim \prod_{p:odd \; prime} \frac{p-(-1)^{\frac{p-1}{2}}}{p} \quad (t \to \infty).
\end{align*}

\end{proof}

Hereafter we explain many examples satisfying DH. We remark that DH is quite difficult in general. For example let $X$ be an abelian variety (e.g. elliptic curve), then DH is the original version of BSD \cite{BS} with $r(X)={\rm rank}X({\mathbb Q})$ and it will imply the Riemann hypothesis for the associated $L$-function $L(s,X)$ as indicated by \cite{G} (at least for dim$(X)=1$.) We remark that the Deep Riemann Hypothesis is studied in \cite{KK,KKK}.

\begin{theorem}[${\rm GL}(n)$] 
\begin{align*}
C({\rm GL}(n))&=e^{-\gamma} \prod_{k=2}^{n} \zeta(k)^{-1}.\\
r({\rm GL}(n))&=-1.
\end{align*}
\end{theorem}

\begin{theorem}[${\rm SL}(n)$]
\begin{align*}
C({\rm SL}(n))&= \prod_{k=2}^{n} \zeta(k)^{-1}.\\
r({\rm SL}(n))&=0.
\end{align*}
\end{theorem}

\begin{theorem}[${\rm Sp}(n)$]
\begin{align*}
C({\rm Sp}(n))&=\prod_{k=1}^{n} \zeta(2k)^{-1}.\\
r({\rm Sp}(n))&=0.
\end{align*}
\end{theorem}

\begin{theorem}[${\mathbb A}^{n}$]
\begin{align*}
C({\mathbb A}^{n})&=1.\\
r({\mathbb A}^{n})&=0.
\end{align*}
\end{theorem}

\begin{theorem}[${\mathbb P}^{n}$]
\begin{align*}
C({\mathbb P}^{n})&=e^{\gamma} \zeta(n+1)^{-1}.\\
r({\mathbb P}^{n})&=1.
\end{align*}
\end{theorem}

\begin{theorem}[${\rm Gr}(n,m):n>m>1$]
\begin{align*}
C({\rm Gr}(n,m))&=e^{\gamma} \frac{\prod_{k=2}^{m} \zeta(k)}{\prod_{k=n-m+1}^{n} \zeta(k)}.\\
r({\rm Gr}(n,m))&=1.
\end{align*}
\end{theorem}

For a monic polynomial $f(x) \in {\mathbb Z}[x]$ we define 
\[
||f||_{t} =\prod_{p \leq t} \frac{f(p)}{p^{{\rm deg}(f)}}
\]
and study the property
\[
||f||_{t} \sim C(f) (\log t)^{r(f)}
\]
as $t \to \infty$. Then Theorems 1-6 are essentially reduced to the case of the cyclotomic polynomial $\Phi_{n}$.

\begin{theorem}
\begin{align*}
C(\Phi_{n})&=e^{- \gamma \mu(n)} \prod_{d|n \atop d>1} \zeta(d)^{-\mu(\frac{n}{d})}.\\
\gamma(\Phi_{n})&=-\mu(n).
\end{align*}
\end{theorem}

Now we recall a Serre's estimate \cite{S} concerning Lang-Weil estimate \cite{LW}.

\begin{theoremC} Let $X$ be an algebraic variety over the rational number field ${\mathbb Q}$. Then we have
\[
\biggl| |X({\mathbb F}_{p})| - p^{{\rm dim}(X)} \biggl| \leq B p^{{\rm dim}(X)-\frac{1}{2}},
\]
where $B$ is a constant independent of $p$.
\end{theoremC}

We notice that $\biggl| |X({\mathbb F}_{p})| - p^{{\rm dim}(X)} \biggl| \leq B p^{{\rm dim}(X)-\frac{1}{2}}$ can be written as
\[
\biggl| \frac{|X({\mathbb F}_{p})|}{p^{{\rm dim}(X)}} - 1 \biggl| \leq \frac{B}{\sqrt{p}}.
\]

Let $A(p)$ be a numerical sequence satisfying
\[
\lim_{p \to \infty} \frac{A(p)}{p^{d}}=1
\]
with $d \in {\mathbb Z}_{\geq 0}$. Then we define 
\[
b(p)=\sqrt{p} \biggl(\frac{A(p)}{p^{d}}-1\biggl),
\]
that is, 
\[
\frac{A(p)}{p^{d}}=1+\frac{b(p)}{\sqrt{p}}.
\]

We notice that by Theorem C $b(p)$ is finite ($|b(p)| \leq B$) if $b(p)=b_{X}(p)$ with $A(p)=|X({\mathbb F}_{p})|$. 

\begin{theorem}[${\mathbb P}^{n}$] Let $X={\mathbb P}^{n}$. Then
\[
b_{X}(p)=\frac{1}{\sqrt{p}} \frac{1-p^{-n}}{1-p^{-1}} (>0).
\]
\end{theorem}

\begin{theorem}[${\mathbb A}^{n}$] Let $X={\mathbb A}^{n}$. Then
\[
b_{X}(p)=0.
\]
\end{theorem}

\begin{theorem}[${\rm GL}(1)$] Let $X={\rm GL}(1)$. Then
\[
b_{X}(p)=-\frac{1}{\sqrt{p}} (<0)
\]
\end{theorem}

\begin{theorem}[${\rm GL}(2)$] Let $X={\rm GL}(2)$. Then
\[
b_{X}(p)=-\frac{1}{\sqrt{p}} -\frac{1}{p\sqrt{p}} +\frac{1}{p^{2}\sqrt{p}} (<0)
\]
\end{theorem}

\begin{theorem}[${\rm SL}(2)$] Let $X={\rm SL}(2)$. Then
\[
b_{X}(p)=-\frac{1}{p\sqrt{p}} (<0)
\]
\end{theorem}

The following theorem gives an example where $b(p)$ is not necessarily finite. 

\begin{theorem} Let $A(p)=p^{d}+p^{d-\frac{1}{3}}$. Then $b(p)$ is not finite.
\end{theorem}

Finally, we calculate $b_{X}(p)$ for elliptic curve $X$ over ${\mathbb Q}$ with $A(p) = |X({\mathbb F}_{p})|$.

\begin{theorem} For sufficiently large $p$ ($p$ is ``good'') we have
\[
-2 < b_{X}(p)<3.
\]
\end{theorem}

\section{Proof of Main results} 

\begin{proof}[Proof of Theorem 1] 
Using 
\[
|{\rm GL}(1,{\mathbb F}_{p})|=p-1
\]
and Theorem A, we have
\begin{align*}
||{\rm GL}(1)||_{t} &= \prod_{p \leq t} \frac{|{\rm GL}(1,{\mathbb F}_{p})|}{p}\\
&= \prod_{p \leq t} \{(1-p^{-1})\}\\
&\sim e^{-\gamma} \cdot (\log t)^{-1} \quad (t \to \infty).
\end{align*}
Let $n \geq 2$. Using 
\[
|{\rm GL}(n,{\mathbb F}_{p})|=p^{n^{2}}(1-p^{-1})(1-p^{-2}) \cdots (1-p^{-n})
\]
and Theorem A, we have
\begin{align*}
||{\rm GL}(n)||_{t} &= \prod_{p \leq t} \frac{|{\rm GL}(n,{\mathbb F}_{p})|}{p^{n^2}}\\
&= \prod_{p \leq t} \{(1-p^{-1})(1-p^{-2}) \cdots (1-p^{-n})\}\\
&= \prod_{p \leq t} (1-p^{-1}) \prod_{p \leq t} \{(1-p^{-2}) \cdots (1-p^{-n})\}\\
&\sim e^{-\gamma} \prod_{k=2}^{n} \zeta(k)^{-1} \cdot (\log t)^{-1} \quad (t \to \infty).
\end{align*}
\end{proof}

\begin{proof}[Proof of Theorem 2] 
Using 
\begin{align*}
|{\rm SL}(n,{\mathbb F}_{p})| &= \frac{|{\rm GL}(n,{\mathbb F}_{p})|}{p-1} \\
&=\frac{p^{n^{2}}(1-p^{-1})(1-p^{-2}) \cdots (1-p^{-n})}{p-1}\\
&=p^{n^2 - 1} (1-p^{-2}) \cdots (1-p^{-n}),
\end{align*}
we have
\begin{align*}
||{\rm SL}(n)||_{t} &= \prod_{p \leq t} \frac{|{\rm SL}(n,{\mathbb F}_{p})|}{p^{n^2 - 1}}\\
&= \prod_{p \leq t} \{(1-p^{-2}) \cdots (1-p^{-n})\}\\
&\sim \prod_{k=2}^{n} \zeta(k)^{-1} \quad (t \to \infty).
\end{align*}
\end{proof}

\begin{proof}[Proof of Theorem 3]
Using 
\[
|{\rm Sp}(n,{\mathbb F}_{p})|=p^{n(2n+1)} (1-p^{-2})(1-p^{-4}) \cdots (1-p^{-2n}),
\]
we have
\begin{align*}
||{\rm Sp}(n)||_{t} &= \prod_{p \leq t} \frac{|{\rm Sp}(n,{\mathbb F}_{p})|}{p^{n(2n+1)}}\\
&= \prod_{p \leq t} \{(1-p^{-2})(1-p^{-4}) \cdots (1-p^{-2n})\}\\
&\sim \prod_{k=1}^{n} \zeta(2k)^{-1} \quad (t \to \infty).
\end{align*}
\end{proof}

\begin{proof}[Proof of Theorem 4]
Using 
\[
{\mathbb A}^{n}({\mathbb F}_{p})=({\mathbb F}_{p})^{n},
\]
we have
\begin{align*}
||{\mathbb A}^{n}||_{t} &= \prod_{p \leq t} \frac{|{\mathbb A}^{n}({\mathbb F}_{p})|}{p^{n}}\\
&= 1\\
&\sim 1 \quad (t \to \infty).
\end{align*}
\end{proof}

\begin{proof}[Proof of Theorem 5]
Using 
\[
|{\mathbb P}^{n}({\mathbb F}_{p})|=1+p+\cdots+p^{n}=\frac{p^{n+1}-1}{p-1}
\]
and Theorem A, we have
\begin{align*}
||{\mathbb P}^{n}||_{t} &= \prod_{p \leq t} \frac{|{\mathbb P}^{n}({\mathbb F}_{p})|}{p^{n}}\\
&= \prod_{p \leq t} \frac{1-p^{-(n+1)}}{1-p^{-1}}\\
&\sim \frac{e^{\gamma}}{\zeta(n+1)} \cdot \log t \quad (t \to \infty).
\end{align*}
\end{proof}

\begin{proof}[Proof of Theorem 6]
Using 
\[
|{\rm Gr}(n,m)({\mathbb F}_{p})|=\frac{(p^{n}-1)\cdots(p^{n-m+1}-1)}{(p^{m}-1)\cdots(p-1)}
\]
and Theorem A, we have
\begin{align*}
||{\rm Gr}(n,m)||_{t} &= \prod_{p \leq t} \frac{|{\rm Gr}(n,m)({\mathbb F}_{p})|}{p^{m(n-m)}}\\
&= \prod_{p \leq t} \frac{(1-p^{-(n-m+1)})\cdots(1-p^{-n})}{(1-p^{-1})\cdots(1-p^{-m})}\\
&\sim e^{\gamma}\frac{\zeta(2) \cdots \zeta(m)}{\zeta(n-m+1) \cdots \zeta(n)} \cdot \log t \quad (t \to \infty).
\end{align*}
\end{proof}

\begin{proof}[Proof of Theorem 7]
Using
\[
\Phi_{n}(t)=\prod_{d|n}(t^{d}-1)^{\mu(\frac{n}{d})}
\]
and Theorem A, we have
\begin{align*}
||\Phi_{n}||_{t} &=\prod_{p \leq t} \frac{\Phi_{n}(p)}{p^{{\rm deg}(\Phi_{n})}}\\
&= \prod_{p \leq t}\frac{\prod_{d|n}(p^{d}-1)^{\mu(\frac{n}{d})}}{p^{\varphi(n)}}\\
&= \prod_{p \leq t}\frac{\prod_{d|n}(p^{d}-1)^{\mu(\frac{n}{d})}}{p^{\sum_{d|n}\mu(\frac{n}{d}) d}}\\
&= \prod_{p \leq t} \prod_{d|n} (1-p^{-d})^{\mu(\frac{n}{d})}\\
&= \prod_{p \leq t} (1-p^{-1})^{\mu(n)} \prod_{d|n \atop d>1} (1-p^{-d})^{\mu(\frac{n}{d})} \\
&\sim (e^{-\gamma} (\log t)^{-1})^{\mu(n)} \prod_{d|n \atop d>1} \zeta(d)^{-\mu(\frac{n}{d})} \quad (t \to \infty)\\
&= e^{-\gamma \mu(n)} \prod_{d|n \atop d>1} \zeta(d)^{-\mu(\frac{n}{d})} \cdot (\log t)^{-\mu(n)}.
\end{align*}

\end{proof}

\begin{proof}[Proof of Theorem 8] Since
\[
|X({\mathbb F}_{p})|=p^{n}+p^{n-1}+\cdots+1,
\]
we have
\begin{align*}
\frac{|X({\mathbb F}_{p})|}{p^{n}} &= 1+\frac{1}{p}+\cdots+\frac{1}{p^{n}}\\
&= 1+\frac{b_{X}(p)}{\sqrt{p}}.
\end{align*}
Thus we have
\begin{align*}
b_{X}(p) &= \sqrt{p} (\frac{1}{p}+\frac{1}{p^2}+\cdots+\frac{1}{p^{n}})\\
&= \frac{1}{\sqrt{p}} (1+\frac{1}{p}+\cdots+\frac{1}{p^{n-1}})\\
&= \frac{1}{\sqrt{p}} \frac{1-p^{-n}}{1-p^{-1}}.
\end{align*}

\end{proof}

\begin{proof}[Proof of Theorem 9] Since
\[
|X({\mathbb F}_{p})|=p^{n},
\]
we have
\begin{align*}
b_{X}(p) &= \sqrt{p} (\frac{|X({\mathbb F}_{p})|}{p^{n}} - 1)\\
&= 0.
\end{align*}
\end{proof}

\begin{proof}[Proof of Theorem 10] Since
\[
|X({\mathbb F}_{p})|=p-1,
\]
we have
\begin{align*}
\frac{|X({\mathbb F}_{p})|}{p} &= \frac{p-1}{p}\\
&= 1-\frac{1}{p}\\
&= 1+\frac{b_{X}(p)}{\sqrt{p}}.
\end{align*}
Thus we have
\[
b_{X}(p) = -\frac{1}{\sqrt{p}}. 
\]
\end{proof}

\begin{proof}[Proof of Theorem 11] Since
\[
|X({\mathbb F}_{p})|=p^{4}(1-p^{-1})(1-p^{-2}),
\]
we have
\begin{align*}
\frac{|X({\mathbb F}_{p})|}{p^{4}} &= (1-p^{-1})(1-p^{-2})\\
&=1-p^{-1}-p^{-2}+p^{-3}\\
&= 1+\frac{b_{X}(p)}{\sqrt{p}}.
\end{align*}
Thus we have
\[
b_{X}(p) =-\frac{1}{\sqrt{p}}-\frac{1}{p\sqrt{p}}+\frac{1}{p^{2}\sqrt{p}}.
\]
\end{proof}

\begin{proof}[Proof of Theorem 12] Since
\[
|X({\mathbb F}_{p})|=p^{3}(1-p^{-2}),
\]
we have
\begin{align*}
\frac{|X({\mathbb F}_{p})|}{p^{3}} &= 1-p^{-2}\\
&=1+\frac{b_{X}(p)}{\sqrt{p}}.
\end{align*}
Thus we have
\[
b_{X}(p) =-\frac{1}{p\sqrt{p}}.
\]
\end{proof}

\begin{proof}[Proof of Theorem 13] Since
\begin{align*}
b(p) &= \sqrt{p} (\frac{A(p)}{p^{d}}-1)\\
&=\sqrt{p} (\frac{p^{d}+p^{d-\frac{1}{3}}}{p^{d}}-1)\\
&=p^{\frac{1}{6}},
\end{align*}
we have
\[
\lim_{p \to \infty} b(p) = \infty.
\]
\end{proof}

\begin{proof}[Proof of Theorem 14] For
\[
A(p) = |X({\mathbb F}_{p})| = p+1-a(p)
\]
using Hasse's theorem on elliptic curves we can write
\[
a(p) = 2 \sqrt{p} \cos(\theta(p))
\]
with $\theta(p) \in [0,\pi]$. So we obtain
\begin{align*}
b_{X}(p) &= \sqrt{p} (\frac{A(p)}{p}-1)\\
&=\frac{1}{\sqrt{p}} - 2\cos(\theta(p)).
\end{align*}
Since $-2 \leq 2\cos(\theta(p)) \leq 2$, we have
\begin{align*}
&b_{X}(p) \leq \frac{1}{\sqrt{2}}+2<3,\\
&b_{X}(p) > -2\cos(\theta(p)) \geq -2.
\end{align*}
\end{proof}

\end{document}